\documentclass{elsarticle}

\title{Sorting using non-binary comparisons}

\author[mem]{Richard A. B. Johnson}

\author[ceu]{Gabor Meszaros}

\address[mem]{University of Memphis, Memphis TN, USA, rabjohnson@gmail.com}
\address[ceu]{Central European University, Budapest, Hungary, meszaros\textunderscore gabor@phd.ceu.edu}

\usepackage{amssymb}
\usepackage{amsmath}
\usepackage{amsthm}
\usepackage{hyperref}
\usepackage{color}
\usepackage{fmtcount}
\usepackage{tikz}
\usetikzlibrary{shapes}
\usetikzlibrary{decorations.markings}
\include{longdiv}
\usepackage{multirow}
\hypersetup{colorlinks=true, linkcolor=red}

\newcounter{apple}
\newcounter{pear}

\theoremstyle{definition}
\newtheorem{Theorem}{Theorem}
\newtheorem{Lemma}[Theorem]{Lemma}

\newtheorem{Observation}[Theorem]{Observation}

\newtheoremstyle{mystyle}
  {3pt}{3pt}{}{}{\bfseries}{.}{.5em}
  {\thmname{#1}\ \theapple.\thmnumber{#2}\thmnote{. #3}}
\theoremstyle{mystyle}

\newtheorem{stage}[pear]{Stage}

\newenvironment{method}[0]{\textit{Method:}}{\endproof}

\begin{document}

\begin{abstract}
Given a set of $n$ elements we investigate how much of the ordering can be determined by an instrument that takes $k$ elements and returns the $t_1^\text{st}, t_2^\text{nd}, \ldots, t_s^\text{th}$ of them. We consider this question in both an on-line sense, in which future choices can depend on previous results, and off-line where all the queries must be chosen initally before knowing any results. 
\end{abstract}

\begin{keyword}
sorting \sep scales \sep offline algorithm
\end{keyword}

\maketitle

Our aim in this paper is to study the following question. Assume a user has an ordered set of $n$ elements in which the ordering is fixed but not known (for example distinguishable but unmarked coins of unknown distinct weights), and that he wishes to determine the ordering. We denote this base set $X$, containing elements $x_1$, $x_2 \ldots x_n$. He is given a scale that accepts as input a $k$-set of elements and returns a fixed subset of them according to the ordering, for example it might return a subset of size $s$ that contains the $t_1^\text{st}, t_2^\text{nd}, \ldots, t_s^\text{th}$ elements. We would call such a scale a $(k, t_1, \ldots, t_s)$ scale, and wish to know what one can determine about the ordering of the elements from repeated use of such a scale. We shall refer to the process of using the scale on a $k$-set as \emph{querying} that set. These have been previously studied in the case where $k=2$, in which case they are known as \emph{binary} scales. In this paper we expand this to consider the case where $k$ is greater than 2, and give algorithms to efficiently determine as much of the ordering as possible. We analyse the order of the number of queries for fixed $k$ and large $n$ in both the online and offline settings.\\

Clearly the user cannot completely discover the ordering, as you cannot determine the ordering of the first $t_1-1$ elements or the final $(n-t_s)$ elements, if $t_1 > 1$ and $(n - t_s) < 1$. Let us call these sets $S$ for the initial segment of `small' elements and $L$ for the final segment of `large' elements. Additionally note that if the scale is symmetric (in the sense that the elements that it returns are symmetric around the midpoint of $k$ i.e. $t_1 = k-t_s, t_2 = k - t_{s-1}$ etc) then the ordering cannot be fully determined for the remaining elements, as the results of any query would be the same if the ordering was reflected. We therefore ignore this case, and assume assymmetric instruments in general. We also assume that the scale returns an unordered set, $\{t_1, \ldots, t_s\}$, rather than an ordered one. This is because we show that even with an unordered set you can recover the full ordering of the elements; an ordered output would be strictly stronger, and so also able to do the same.\\

A related question was considered by Hannasch, Kim and McLaughlin \cite{HKL10} in 2010. They considered an instrument which again accepted an input of $k$-elements but returned the complete ordering of the input set. They asked, given such an instrument, how long it would take to determine to first $t$ elements of an $n$-set, which they called $S(n, k, t)$. Although related their instrument is stronger, and the question asked weaker, so we consider our problem a more general version of theirs. We do not know of any prior work done on this, stronger, question.\\

This paper is structured as follows. In section \ref{OnlineSection} we consider online algorithms, where the sets submitted to future queries can depend on past results. Beginning with the case $s=1$, i.e. where the scales output a single element, we show that it is possible to determine the ordering of the elements (excluding the ordering of $S$ and $L$) in $O(n \log n)$ time, the same order of bound as in the binary case. We give this constant explicitly, showing that it is an improvement over the binary case. We also investigate the case where $s>1$, and give an explicit algorithm for determining the ordering in this case.\\ 

In section \ref{OfflineSection} we consider offline algorithms, where all the queries must be specified in advance and then the full set of results are returned simultaneously. This could be applicable in a situation where queries have to be sent to a laboratory to run overnight. Obviously this requires more queries in general. In the case where $s=1$, i.e. where we are using a $(k, t)$-scale, we outline an algorithm that works in $O(n^{k-(t-1)})$ queries, and show that this is the best possible order. This algorithm relies heavily on a recursive approach, determining the results of queries that have not been carried out from those that have. We also outline an alternative algorithm that works in a similar amount of time, but works directly, determining the ordering of the elements using an adjacency based argument. In the case $s>1$ the recursive approach can often still be applied, but the calculations involved get more complicated and require more detailed case analysis. However the adjacency argument continues to work, giving us a general offline algorithm that works when the scale outputs an unordered set.

\section{Online Algorithms}
\label{OnlineSection}

In this section we consider online algorithms, where the user is given the result of each query as he requests it, and on that basis selects the next set of $k$ elements that he wants to query. As highlighted in the introduction, these scales cannot in general determine the full ordering of the element set, as the first and last segments of the ordering will never be returned by any query, and so no query can determine their order. We refer to these segments as $S$ and $L$ respectively, as mentioned in the introduction, and in each case shall highlight which elements they comprise. The remaining elements we denote $X'$, and note that there are at least $n - (k-1)$ of them, which we call $n'$.

We begin by considering a singleton-scale, which returns a singleton output. 

\subsection{Singleton Output Scales}

The classical version of this question is the binary scale which accepts as input two elements and returns the smaller. We consider more general $(k, t)$ scales, accepting $k$ elements and returning the $t^\text{th}$ smallest. To ease notation throughout this section we shall assume that $t \leqslant k/2$, i.e. it lies in the first half of the queried set. If $t > k/2$ then the following analysis still holds, inverting the roles of $S$ and $L$ and making occasional other similarly minor adjustments to the calculations. \\

We describe an algorithm that works in $O\big(n' \log n'\big)$ queries to determine the ordering of the element set. This algorithm works by first determining $S$ and $L$ (in Stages 1.1 and 1.2), and then using them to iteratively determine the ordering by repeatedly determining the smallest element among the remaining unsorted elements (in Stage 1.3). Stage 1.3 shall take the longest, it is this stage that takes $O\big(n' \log n'\big)$ time to run, while Stage 1.1 is linear in $n'$ and Stage 1.2 is independent of $n'$, taking a number of queries that is just a function of $k$. Hence the combination of the three works in $O\big(n' \log n'\big)$ queries as required.\\

\setcounter{apple}{1}

\begin{stage}
Determine the elements that comprise $S \cup L$.
\end{stage}

\begin{method}
Note that these elements are precisely those that can not be returned by any query. Hence they can be identified by eliminating all the others. The user repeatedly picks a $k$-set from those that he has not yet eliminated, and eliminates whatever is the output. Each time he does this it eliminates 1 more element. He continues doing so until he cannot find another $k$-set, which occurs when there are $k-1$ elements left. But note that $S \cup L$ is always contained within the remaining elements, and there are $k-1$ elements in $S \cup L$, so it is exactly the remaining elements at this point.\\
\end{method}

\begin{stage}
Partition $S \cup L$ into $S$ and $L$, and if possible identify which is which.\\
\end{stage}

\begin{method}
Pick an arbitrary set from the eliminated elements, which we denote $a_1, \ldots, a_{k-1}$ where the index respects the ordering of the elements. For each element in $S \cup L$ query it taken together with $\{a_1, \ldots, a_{k-1}\}$. If it was in $S$ this will return $a_{t-1}$ and if it was in $L$ it will return $a_t$. Although the user does not know which is which, he can partition $S \cup L$ into $S$ and $L$ according to which response he gets. He can further keep count of how often each comes up, as he will receive $a_{t-1}$ $|S|$ times, and $a_t$ $|L|$ times. If $|S| \neq |L|$ then this further identifies $S$ and $L$ -- the only time this won't work is when $|S| = |L|$. This occurs when $t = k/2$, which means that the underlying scale is symmetric. Hence by the end of Stage 1.2 in the case of an asymmetric instrument the user knows $S$ and $L$, and if the instrument is symmetric then he has identified the set $\{S, L\}$ but does not know which is which.\\
\end{method}

\begin{stage}
Use $S$ to determine the order of the remaining elements.\\
\end{stage}

\begin{method}
First, let us assume that our instrument is asymmetric, and hence the user knows $S$ before starting this Stage -- at the end we shall address what is done in the symmetric case. In this stage the user shall repeatedly use $S$ to determine the smallest element of a subset of $k' := k-(t-1)$ elements by querying them taken together with $S$. If he wishes to find the smallest element of a smaller set he takes it together with $S$ and supplements it with elements from $L$ until he has $k$ elements allowing him to query it. In both cases as $S$ takes up the first $t-1$ elements of the queried set, it will return the next largest which is the smallest of the subset he is trying to check.\\

The user now take the remaining $n'$ elements and partitions them into as many sets of size $k'$ as possible, with a remainder set of at most $k'$ elements, which we shall call level 1 sets. He then groups the level 1 sets into sets of size $k'$ (again as far as possible, with perhaps a deficient remainder set) to get level 2 sets, so a level 2 set consists of $k'$ sets each of $k'$ elements. He continues this process until all the $n'$ elements are in a single set - we denote the level where this occurs as $d$, where $d = \log_{k'} n'$This effectively creates a $d$-dimensional grid of sets. We illustrate this grouping for the first two layers in Figure \ref{ScalesGrouping}. \\

\begin{figure}[ht] \centering
\begin{tikzpicture}
\tikzstyle{vertex} = [draw, fill=black, minimum size=3pt, inner sep=0pt, shape=ellipse]
% This first foreach loop determines the centre points of each L_2 group
\foreach \centrex/\centrey in {-4/0, 0/0}
{
\draw[fill=black!5] (\centrex, \centrey) ellipse (1.75 and 2);
\node[font=\Large] at (\centrex, -2.5) {$L_2$};
% The second foreach determines the centre of each L_1 within an L_2
\foreach \addx/\addy in {0/1.15, 0.866/-0.35, -0.866/-0.35}
	{
	\draw[fill=black!20] (\centrex + \addx,\centrey + \addy) ellipse (0.6 and 0.6);
	\node[font=\large] at (\centrex + \addx, \centrey + \addy -0.9) {$L_1$};
	% The final foreach adds the little vertices inside each L_1
	\foreach \x/\y in {0/0.4, 0.346/-0.2, -0.346/-0.2}
		{
		\node[vertex] at (\centrex + \addx + \x,\centrey + \addy + \y) {};
		}	
	}
}
% Now we add the 3rd set which has deficiencies
\draw[fill=black!5] (4, 0) ellipse (1.75 and 2);
\node[font=\Large] at (4, -2.5) {$L_2$};
\draw[fill=black!20] (4 + 0,0 + 1.15) ellipse (0.6 and 0.6);
\node[font=\large] at (4 + 0, 0 + 1.15 -0.9) {$L_1$};
\draw[fill=black!20] (4 + -0.866,0 + -0.35) ellipse (0.6 and 0.6);
\node[font=\large] at (4 + -0.866, 0 + -0.35 -0.9) {$L_1$};
% We have a full first L_1 set
\foreach \addx/\addy in {0/1.15}
	{
	% The final foreach adds the little vertices inside each L_1
	\foreach \x/\y in {0/0.4, 0.346/-0.2, -0.346/-0.2}
		{
		\node[vertex] at (4 + \addx + \x,0 + \addy + \y) {};
		}	
	}
% And a deficient second L_1 set
\node[vertex] at (4 + -0.866 + 0,0 + -0.35 + 0.4) {};
\end{tikzpicture}
\caption{First two layers of the grouping process}
\label{ScalesGrouping}
\end{figure}
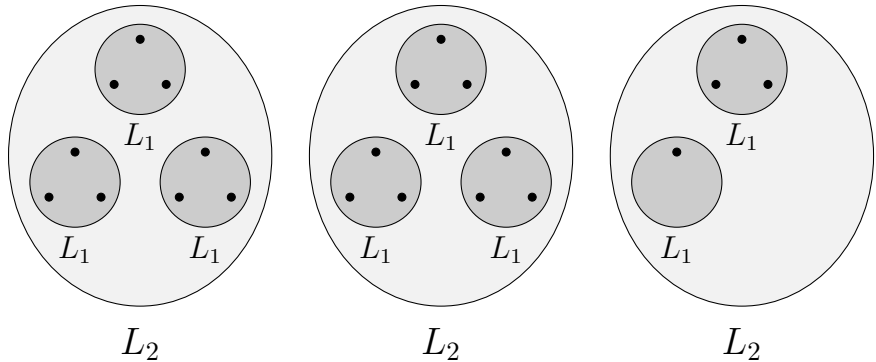

At the first step he queries all the level 1 sets, establishing which is the smallest element of each. He then `checks' each level 2 set in turn by querying the smallest element of each level 1 set inside it to find the smallest element in each level 2 set. Continuing in this manner he can establish the smallest element in every level $r$ set for all $1 \leq r \leq d$, and hence the smallest element in the level $d$ set. But the level $d$ set contains everything, so he  has found the smallest element in the remaining $n'$ elements. He now wants to remove this element and find the next smallest. Note that this only requires him to check/query the level $1, 2, \ldots, d$ sets that the previous smallest was in, as the others are unaffected. Repeating $n'$ times therefore determines the order of the remaining elements.\\

This completes the analysis of the asymmetric case. If the instrument is symmetric, then recall as discussed earlier that the user would only at best be looking to determine the ordering or its reflection, since he would not be able to distinguish these with any query. If his instrument is symmetric though the main difference is that at the start of Stage 1.3 he wouldn't know which set was $S$ and which was $L$. However he can arbitrarily assume either one is $S$, and carry out Stage 1.3 under that assumption. If he was correct he will get the correct ordering, if incorrect he will get the reflection, and given that he cannot distinguish these with a symmetric instrument anyway that therefore gives him the best possible information he could gather about the ordering.\\
\end{method}

It remains to show that this method only uses $O\big(n' \log(n')\big)$ queries. Stage 1.1 requires one query for each element in $X \setminus (S \cup L)$, so uses $n-(k-1)$ queries. Stage 1.2 requires one query for each element in $S \cup L$, so uses $(k-1)$ queries. Thus together Stages 1.1 and 1.2 use a total of $n$ queries.\\

Stage 1.3 takes longer, this is where the extra log factor comes in. The first run through of the levels in which the user finds the smallest value in each set takes this many queries

\begin{equation*}
\sum_{i=1}^d\left\lceil\frac{n'}{(k')^i}\right\rceil \leqslant d n'
\end{equation*}

Having completed these queries, and thus found the smallest element in the set $X \setminus (S \cup L)$, the user then needs to carry out $d$ additional queries for each remaining element. Hence he requires at most $d (n'-1)$ remaining queries -- in practice he may require slightly less as when he gets some levels with only 1 element he doesn't need to continue to query them.\\
%This may bear some consideration actually -- he can only query each level set $k'$ times, as each time he only does so if it lost an element. So maybe we can improve this? I doubt it'll be significant though.

Hence in total this algorithm required at most the following number of queries

\begin{equation*}
n + 2 d n' = n + 2 n' \log_{k'} (n') = O\big(n' \log n'\big).
\end{equation*}

\subsection{Multiple output instruments}

We now turn our attention to scales that return multiple elements. These accept as input $k$ elements from $X$ and return a set of size $s$ containing the $t_1^\text{st}$, $t_2^\text{nd} \ldots t_s^\text{th}$ elements. We refer to such a scale as a $(k, t_1, \ldots, t_s)$ scale.\\

In this case we shall outline an algorithm that works in similar stages to the singleton output case, although in each case more work shall be needed to achieve the same ends. We shall again make use of the initial and final segments, which with slight abuse of notation we again denote by $S$ and $L$. This time however $S$ consists of the first $t_1 - 1$ elements, and $L$ of the final $k - t_s$ elements. Together they therefore comprise a set of $k - 1 - (t_s - t_1)$ elements. It is possible if $n$ is small compared to $k$ that there can be other elements that the user cannot distinguish the order of. If, for example, we consider a $(7, 2, 6)$ instrument on 8 elements then none of $\{x_1, x_4, x_5, x_8\}$ will ever included in an output, so as well as the initial and final segments we have a middle segment that is indistinguishable. However in our case we think of $n$ as being arbitrarily large compared to $k$; indeed we consider the asymptotics as $n \rightarrow \infty$, and if $n > 2k$ then no such middle set of indistinguishable elements can exist.\\

\stepcounter{apple}
\setcounter{pear}{0}

\begin{stage}
Determine the elements that comprise $S \cup L$.\\
\end{stage}

\begin{method}
Again, as in the singleton case, $S \cup L$ comprise the elements that are never in the output of any query. So the user can begin by repeatedly querying uneliminated elements, and discarding the elements contained in the outputs. In the singleton case, this worked until there were $k-1$ elements left, and these comprise exactly $S \cup L$. In this case however the user can only eliminate elements until he has $k-s$ left, and in general $S \cup L$ have $k - 1 - (t_s - t_1)$ elements. These are the same if $t_s - t_1 = s - 1$, which occurs if the numbers $t_1$, $t_2 \ldots, t_s$ are consecutive. In that case Stage 2.1 terminates at this point.\\

If $t_1, \ldots, t_s$ are not consecutive then the user has a little more work to do. After carrying out the above he has $k - s$ candidates left for $S \cup L$, which contains $S \cup L$ and some extra elements which he wishes to eliminate. To do so we use an inductive approach. Specifically we shall show that if we have a set of $k-a$ candidates with $a < s$, $\{c_1, \ldots, c_{k-a}\}$, that contains at least 1 element not in $S \cup L$ then we can eliminate 1 more. This clearly suffices to eliminate all those not in $S \cup L$. Note that if we query $\{c_i\}$ along with a set of size $a$, as we get at least $a$ outputs they must either consist exactly of the $a$ additional elements, or include and thus eliminate an additional candidate. So to achieve the inductive aim, pick any set of $(2a - 1)$ already eliminated elements, which we denote $\{e_1, \ldots, e_{2a-1}\}$, and carry out all the $\binom{2a-1}{a}$ queries involving $a$ of the additional elements along with the candidates. Let $x$ be an element in $\{c_i\}$ not in $S \cup L$, and note that by the pigeonhole principle either at least $a$ of the $\{e_i\}$ are lower than $x$ in the ordering, or at least $a$ of them are greater -- without loss of generality we assume the former, and relabelling if necessary we assume that the set $\{e_1, \ldots, e_a\}$ is among them. Hence when we query $\{c_1, \ldots, c_{k-a}, e_1, \ldots, e_a\}$, all of $L$ and $x$ are greater than all of the additional elements, hence the $t_r^\text{th}$ element cannot be one of the additional elements and so must lie in the $\{c_i\}$, eliminating it. This process can continue until the remaining candidates are exactly $S \cup L$, after this point there does not exist any $x \in \{c_1, \ldots, c_{k-a}\} \setminus (S \cup L)$. Hence this process ends exactly with $S \cup L$ being identified. \\
\end{method}

\begin{stage}
Partition $S \cup L$ into $S$ and $L$, and if possible identify which is which.\\
\end{stage}

\begin{method}
We take exactly the same approach as in the singleton case. Pick any set of size $k-1$ taken from $X \setminus (S \cup L)$, which we denote $a_1, \ldots, a_{k-1}$. The user carries out the $|S \cup L|$ queries consisting of this reference set taken with one element from $S \cup L$. If the element from $S \cup L$ was in $S$ then this returns $\{a_{t_1-1}, \ldots, a_{t_s-1}\}$, while if it was in $L$ it returns $\{a_{t_1}, \ldots, a_{t_s}\}$. Although the user cannot immediately distinguish these they are clearly distinct, so he can partition $S \cup L$ into $S$ and $L$ according to the multiplicites of the responses.\\

In the singleton case the user could also determine which of these sets was $S$ and which was $L$ in the assymmetric case, since that implied they would be of different sizes. This is more complicated in the multiple-output situation, since you could have $t_1 + t_s = k + 1$, giving $|S| = |L|$, but still have an assymetric instrument as a result of some other elements in the output. In practice it doesn't matter which is which, and so for the moment we do not address the question of how to distinguish them in this algorithm. We will however address this at the end of this section for completeness.\\
\end{method}

\begin{stage}
Use $S$ to determine the order of the remaining elements.\\
\end{stage}

\begin{method}
If $S$ and $L$ are of different sizes, then the following works: Let $S'$ be the set of the first $t_s - 1$ elements of $X$, noting that this includes $S$. If the user can identify $S'$ then he can reduce our scale to a $(k', 1)$ scale by insisting on always including $S'$ in any query -- this fills up the first $t_s - 1$ slots of the scale, and means that it will always return some subset of $S'$ (easily ignored) along with the smallest element of the remainder. Defining $X'$ as $X \setminus S'$ he can then use this $(k', 1)$ scale to sort $X' \setminus L$ in $O\big(n \log n \big)$ steps as per Stage 1.3 of the singleton output algorithm. This sorts the majority of $X$ (assuming as always that $n$ is large compared to $k$. To sort the remaining elements in $S' \setminus S$ is then straightforward since the user will have identified the final $k$ elements of $X$, so can create an $(k'', k'')$ instrument by including $k - t_1$ elements of $X$. This can be used to sort the remaining elements.\\

It just remains therefore to say how to find $S'$. If $t_1 - 1$ divides $t_s - 1$ then $S'$ is easy to find by repeatedly removing the smallest elements from $X$. In Stages 2.1 and 2.2 we outlined how to identify $S$ i.e. the smallest $t_1 - 1$ elements in $X$. By removing these and repeating the process the user can identify the next $t_1 - 1$ elements repeatedly until he has found the first $t_s - 1$ elements. If $t_1 - 1$ does not divide $t_s - 1$ then the same process can be applied, except that when the user needs fewer than $t_1 - 1$ more elements at the end to top off $S'$, he leaves in some of the previous $t_1 - 1$ set that he removed so that he requires something of the correct size.\\

If $S$ and $L$ are of the same size then at the end of Stage 2.2 the user has partitioned $S \cup L$ into sets $A$ and $B$, which are $S$ and $L$ but he doesn't know which is which. We suggest the following: he makes an arbitrary assignment, claiming that $A$ is $S$. Discarding this and repeating Stages 2.1 and 2.2 on the remaining $X \setminus A$ elements will return two more sets, $A'$ and $B$. Note that this will be the same $B$ as before, so he can identify $A'$ and discard it again. Continuing in this manner he will eventually find a set that is either $S'$ or $L'$, as above. He can assume this is $S'$, and use it to form what he thinks is an $(k', 1)$ instrument. Using this he can sort the remaining elements as above, and come up with an ordering for $X$. This ordering will either be correct, or exactly the reverse of the correct ordering if his initial theory that $A$ was $S$ was wrong. He can check which of these is true if he had an asymetric instrument by carrying out any query of $k$ elements taken from $X \setminus (S \cup L)$, and seeing if the answer agrees with his opinion on what the ordering is. If it does not, he simply reverses the ordering. \\
\end{method}

It now remains to analyse how long this shall take. Stage 2.1 takes $\left\lceil\frac{n-(k-s)}{s}\right\rceil$ queries initially to get down to $k-s$ elements. To get from there to $S \cup L$ involves removing at most $k$ elements, and removing each takes at most $\binom{2k-1}{k}$ queries, so in total this takes at most $(2k)^{k+1}$ queries. Stage 2.2 then takes $|S \cup L| = k - 1 - (t_s - t_1)$ queries. Hence between them Stages 2.1 and 2.2 take in total a linear number of queries in $n$ with an additional number of queries that is solely a function of $k$. As we consider the situation where $n$ is large and $k$ is small and fixed, this is effectively a linear number of queries in $n$.\\

Stage 2.3 takes as most $k$ runs of Stages 2.1 and 2.2 to identify $S'$, which is therefore still linear in $n$. Having identified $S'$ it then takes $O\big(n \log n\big)$ steps to sort the set $X'$. The final sorting of the remanent, and the extra query in the case where $S$ and $L$ are the same size, clearly only take a number of steps that is a function of $k$, hence overall this algorithm also runs in time $O\big(n \log n\big)$ as required.\\

As promised, we now consider how to distinguish $S$ from $L$ in the case of an assymetric instrument in Stage 2.2. Let $p$ be the smallest index that makes the instrument non-symmetric, in the sense that only one of the $p^\text{th}$ and $(k+1-p)^\text{th}$ elements are in the output. Without loss of generality we shall assume it is the case that the $p^\text{th}$ is in the output and its reflection (inside the scale) is not, to ease notation. We define $S_1 = S$, $L_1 = L$ and $X_1 = X \setminus (S \cup L)$, similarly to before, and then recursively define $S_i$, $L_i$ and $X_i$ to be the initial, final and middle segments of $X_{i-1}$ respectively. Hence for any index $i$ $X$ is thus composed of $\bigcup_{j=1}^i S_j \cup \bigcup_{j=1}^i L_j \cup X_i$. Note that given that Stages 2.1 and 2.2 determined $\{S, L\}$ from $X$, they can be repeated to determine $\{S_i, L_i\}$ and $X_i$ from $X_{i-1}$. Hence the user can build up as many pairs of sets as he wants, all of which form the initial and final segments of the remainder of the full set, provided $n$ is sufficiently large.\\

We use this idea to show how to find $S$ and $L$. First the user determines $\{S_i, L_i\}$ for all $1 \leqslant i \leqslant p+k-2$. He can then identify $S_p$ by querying a set containing one element from each of the sets in the pairs $\{S_1, L_1\}$, $\{S_2, L_2\}$, $\ldots$, \{$S_{p-1}, L_{p-1}\}$, a single element from one of the sets in $\{S_p, L_p\}$ and then balancing elements from $X_p$ to fill out the instrument. If the element he picked from $\{S_p, L_p\}$ was from $S_p$ then it will be in the output, if it was from $L_p$ then it won't, enabling him to identify $S_p$ and $L_p$. He can then repeat this to $\{S_{p+1}, L_{p+1}\}$, $\{S_{p+2}, L_{p+1}\}$, $\ldots$, $\{S_{p+k}, L_{p+k}\}$, and thus identify exactly the sets $S_p, S_{p+1}, \ldots, S_{p+k-2}$. Now identifying $S$ is simple, since he just needs to take one element from one of $\{S, L\}$, and query it along with $k-1$ elements, taken one each from $S_p, S_{p+1}, \ldots, S_{p+k-2}$. The result of this query will determine if the element from $\{S, L\}$ was from $S$ or $L$ merely by looking at which of the other elements was returned, and hence he can identify which set is $S$ and which is $L$. As $p$ is at most $k/2$, if he wished to do this it would require at most $3k/2$ additional runs of Stages 2.1 and 2.2, which is still therefore linear in $n$ and so would not materially affect the running time of the algorithm for large $n$.

\section{Offline Algorithms}
\label{OfflineSection}

We now turn our attention to offline algorithms. In this situation the user must specify the full list of queries that he wishes to carry out in advance, and then receives all the answers simultaneously afterwards. As we saw in Section \ref{OnlineSection}, if the user knows the results of all the possible queries then he can determine essentially the full ordering (excluding $S$ and $L$ as before), since he can follow any of the online algorithms, looking up the results of any query from his bank of known query results. We concern ourselves with trying to minimise the number of queries that he must request in order to determine this ordering.\\

We begin by considering a singleton-scale, which returns a singleton output. 

\subsection{Singleton Output Scales}

In this case, as in the online section on singleton output scales, the user is given a $(k, t)$ scale that accepts a $k$-set as input and returns the $t^\text{th}$ smallest element. Again for simplicity we assume that $t \leqslant k/2$, if not then similar analysis follows, occasionally replacing the word `smallest' with `largest' and, where we have comments about filling the scale from the lower elements, instead filling it from the higher. \\

We first note a simple lower bound, namely that if there is any $t$-set that is not included in a query then there are some orderings that the algorithm would not be able to distinguish. This is because if the missed $t$-set comprised the first $t$ elements of the ordering then no element of it would ever be included in the output of any set. Thus the user would not be able to tell which $t-1$ subset of it formed $S$, and which element was the lowest element of $X \setminus S$. The same applies if there was a $k-(t-1)$-tuple that was not included in any query, since it could form the largest $k-(t-1)$ elements, and then the user would not be able to tell which was the largest element of $X \setminus L$. As $t \leqslant k/2$ by assumption, this second set is larger, and so there are more possible $k-(t-1)$ tuples than there are $t$-tuples. Hence this forms the restriction that we appeal to. Noting that there are $\binom{k}{k-(t-1)}$ $(k-(t-1))$-sets in any query, and as there are $\binom{n}{k-(t-1)}$ $k-(t-1)$-sets in total, this means that all algorithms must contain at least the following number of queries

\begin{equation*}
\frac{\binom{n}{k-(t-1)}}{\binom{k}{k-(t-1)}} = O(n^{k-(t-1)})
\end{equation*}

We shall show that, in fact, there are algorithms that use this order of number of queries. We offer two for consideration, one that relies on recursively deducing the results of all possible queries, and thus the ordering, the second of which is direct and relies on determining the adjacencies of the ordering. We begin with the recursive algorithm.

\subsubsection{Recursive Algorithm}

Our algorithm works by fixing some set of $r$-elements, $Y := \{y_1, \ldots, y_r\}$, and requesting all the queries that involve $Y$ and a $(k-r)$ set from $X \setminus Y$. We shall show that, provided $r$ is not too large, then from this the user can deduce the result of an arbitrary query containing any $(r-1)$-subset of $Y$. If this holds then, by induction, the user can deduce the result of any query, and hence the full ordering. We prove this inductive claim by induction on $r$, beginning with the case $r=1$.

\begin{Theorem}
\label{ScalesOneFixed}
If $y$ is a fixed element and the results of all queries including $y$ are known, then the result of a query on any set $\{a_1, \ldots, a_k\}$ can be deduced.
\end{Theorem}

\begin{proof}
Note that the claim is trivial if $y \in \{a_1, \ldots, a_k\}$, as this would mean that this exact query had taken place. So let us assume that $y \notin \{a_1, \ldots, a_k\}$. We wish to deduce the value of $a_t$ from queries of the form $\{y, a_1, \ldots, a_k\} \setminus \{a_i\}$ for $1 \leq i \leq k$. We shall split into 4 cases for $y$ and 3 cases for $a_i$ in relation to $a_t$, and count how often we get various responses. These are summarised in the following grid:

\begin{equation*}
\begin{array}{c|c||c|c|c|c}
 & & \multicolumn{4} {|c} {\text{Response}}\\
 & \text{Multiplicity} & y < a_{t-1} & y \in (a_{t-1}, a_t) & y \in (a_t, a_{t+1}) & y > a_{t+1}\\
\hline
\hline
a_i < a_t & t-1 & a_t & a_t & y & a_{t+1}\\
\hline
a_i = a_t & 1 & a_{t-1} & y & y & a_{t+1}\\
\hline
a_i > a_t & k-t & a_{t-1} & y & a_t & a_t\\
\end{array}
\end{equation*}

Now, when performing these queries we get all the results from some column. So if, for example $y < a_{t-1}$, then we get $a_t$ $(t-1)$ times and $a_{t-1}$ $(1 + (k-t))$ times. We can establish which column we are in, and thus how $y$ compares with $a_t$, by looking at the multiplicities of the answers - if we have two different answers with multiplicities $(t-1)$ and $(k-t+1)$ then we are in one of the first two columns, and if we get multiplicities $t$ and $(k-t)$ then we are in one of the last two columns. Further if we get the answer $y$ for some of our queries we are in the middle two columns, if not we are in the outside columns. Thus we can determine which column we are in. Now by taking the result with the appropriate multiplicity ($(t-1)$ in the first two columns, and $(k-t)$ in the latter two) we can tell the value of $a_t$, as required. We can summarise these in the following associated table:

\begin{equation*}
\begin{array}{c|c|c}
\text{Case} & \text{Multiplicities} & \text{Mult. of }a_t\\
\hline
\hline
y < a_t & (t-1, k-t+1) & t-1\\
\hline
y > a_t & (t, k-t) & k-t
\end{array}
\end{equation*}

\end{proof}

The general case is somewhat tricky to see, as the case analysis gets very detailed. Instead, we present the case for $r=2$, which covers most of the concepts that we appeal to, and then explain how the argument changes for a general $r$. The first key point is that just taking the queries involving $x$ and $y$ would not by itself be enough, as the user will also need to know which is larger out of $x$ and $y$. But the following lemma gives a simple way to do that 

\begin{Lemma}
\label{findpair}
Assume that we have a asymmetric scale. Let $z_1, \ldots, z_{k+1}$ be $(k+1)$ fixed elements of our set. By querying all the $\binom{k+1}{k}$ subsets of them we can find two of them $x$ and $y$ such that neither $x$ nor $y$ are in $S \cup L$ and we know $x < y$.
\end{Lemma}

\begin{proof}
Relabelling these reference elements according to the ordering, we note that any query of a subset of them will either return $z_t$ or $z_{t+1}$. As these are possible responses to queries, neither can be a member of $S$ or $L$ so we take these as our $x$ and $y$. It remains to show that the user can identify which is the smaller, but it is clear that the user will receive the answer $z_t$ $(k+1-t)$ times, and $z_{t+1}$ $t$ times, enabling him to distinguish them if $t \neq \frac{k+1}{2}$. But this must hold, otherwise the scale would be symmetric, which contradicts our assumption.
\end{proof}

We now need the next requirement, that given $x$ and $y$, two fixed elements of the set such that the user knows that $x < y$, and all the queries involving this pair, the user can determine the results of any possible query, and hence as much of the ordering as could ever be possible.

\begin{Theorem}
\label{ScalesTwoFixed}
If $x$ and $y$ are two fixed elements such that $x < y$ and the results of all queries including $x$ and $y$ are known then the result of a query on any set $\{a_1, \ldots, a_k\}$ can be deduced.
\end{Theorem}

\begin{proof}
Note that by Theorem \ref{ScalesOneFixed} it suffices to show that we can find the result of any query involving $x$ and an arbitrary set of other elements, $\{a_1, \ldots, a_{k-1}\}$. We shall proceed on this basis. Relabelling them we can refer to such a set as $\{b_1, \ldots, b_k\}$, noting that $x$ is now one of the $\{b\}$, say $b_i$. We wish to establish how to find which is $b_t$ from a set of queries in which we replace each of the $b$ apart from $b_i$ by $y$. The results for these queries -- depending on whether the $b_j$ that $y$ replaces is smaller than, equal to or larger than $b_t$ -- are summarised in the following table - note that now the multiplicities vary according to the position of $x = b_i$:

\begin{equation*}
\begin{array}{c||c|c|c||c|c|c|cc}
& \multicolumn{3} {|c||} {\text{Multiplicities}} & \multicolumn{4} {|c} {\text{Responses}} \\
& x < b_t & x = b_t & x > b_t & y < b_{t-1} & y \in (b_{t-1}, b_t) & y \in (b_t, b_{t+1}) & y > b_{t+1} \\
\hline
\hline
b_j < b_t & t-2 & t-1 & t-1 & b_t & b_t & y & b_{t+1}\\
\hline
b_j = b_t & 1 & 0 & 1 & b_{t-1} & y & y & b_{t+1}\\
\hline
b_j > b_t & k-t & k-t & k-t-1 & b_{t-1} & y & b_t & b_t\\
\end{array}
\end{equation*}

We can again combine multiplicities according to $y < b_t$ and $y > b_t$. At first this looks like it won't let us differentiate options, as we get some situations with the same multiplicities:

\begin{equation*}
\begin{array}{c|c||c|c}
\multicolumn{2} {c||} {\text{Case}} & \text{Multiplicities} & \text{Mult. of }a_t\\
\hline
\hline
\multirow{3}{*}{$y < b_t$} & x < b_t & (t-2, k-t+1) & t-2\\
\cline{2-4}
 & x = b_t & (t-1, k-t) & t-1\\
\cline{2-4}
 & x > b_t & (t-1, k-t) & t-1\\
\hline
\hline
\multirow{3}{*}{$y > b_t$} & x < b_t & (t-1, k-t) & k-t\\
\cline{2-4}
 & x = b_t & (t-1, k-t) & k-t\\
\cline{2-4}
 & x > b_t & (t, k-t-1) & k-t-1\\
\end{array}
\end{equation*}

However we can see that this doesn't matter. Firstly the $2^\text{nd}$ and $3^\text{rd}$ rows of the above table are impossible, as our initial assumption was that $x < y$, so we can remove those. Secondly, we note that of the remaining 4 situations although 2 have the same multiplicities, in either of those two cases we just take the solution with multiplicity $k-t$ and conclude that this is $b_t$. Hence we can identify $b_t$ for any arbitrary set of elements $\{b_1, \ldots, b_k\}$ containing $x$, and hence by Theorem \ref{ScalesOneFixed} we can determine the order of the full set.
\end{proof}

We now give our method for the case $r=2$. The user first picks a set of size $k+1$, and requests all the queries involving $k$ of those -- note that there are $(k+1)$ of these. By Lemma \ref{findpair} this will find him a pair $x$ and $y$ from this set such that he knows $x < y$. He also considers all possible $2$-tuples from this $k+1$ set, and for each pair requests all queries involving that pair. This means that, in particular, even though he is carries out these queries offline he will know the results of all possible queries involving $x$ and $y$. In total this requires an additional $(k+1) + \binom{k+1}{2} \binom{n-2}{k-2}$ queries, which is of the same order as $\binom{n-2}{k-2}$, and hence $O(n^{k-2})$. But by Theorem \ref{ScalesTwoFixed} from these he can deduce the result of any query involving just one of $x$ and $y$, and hence by Theorem \ref{ScalesOneFixed} he can determine as much of the ordering as he could have hoped.\\

As advertised, we shall not give the full explicit argument for general $r$, as the analysis is tedious and not much more enlightening than the $r=2$ case. We shall instead explain how to modify the $r=2$ case. The first part is simple enough - carrying out all the possible probes on a $k+1$ set guaranteed us a pair of elements $x, y$ such that we knew their internal ordering. In general, carrying out all the probes on some $(k+r-1)$ set guarantees us a set of $r$ elements which we can completely order from these probes -- this is simply seen by just taking those as our whole universe and using any argument such as that outlined in the online cases. \\

The notation for the other part gets more involved. We require the following statement by induction. Let $\{x_1, \ldots, x_r\}$ be our reference set which we know the complete ordering of. We want to be able to say that we can deduce the value of some query involving the first $(r-1)$ of the reference set and some set of elements $\{a_1, \ldots, a_{k-r+1}\}$ by considering all the queries containing the full $r$ elements of the reference set and some $(k-r)$ subset of the $a_i$'s. Let us relabel the set $\{x_1, \ldots, x_{r-1}, a_1, \ldots, a_{k-r+1}\}$ as $\{b_1, \ldots, b_k\}$, so that some of the $b$ are taken from $x$-elements and some from $a$-elements. We shall then consider all the queries in which we replace one of the $a$-elements by $x_r$. Note that when counting multiplicities we must consider where $b_t$ lies relative to our reference set -- i.e. how many of them are below it, and whether or not one of them is $b_t$. This gives rise to the following table of multiplicities, we have omitted the left-hand four columns as they are again the same as the above.

\begin{equation*}
\begin{array}{c||c|c|c|c|c|c|c|c|c|}
& \multicolumn{4} {|c|} {\text{\# of reference set smaller than } b_t} &\\
& (r-1)< b_t & (r-1) \leq b_t & (r-2) < b_t & (r-2) \leq b_t & \cdots\\
\hline
\hline
b_j < b_t & t-r & t-r+1 & t-r+1 & t-r+2 & \cdots\\
\hline
b_j = b_t & 1 & 0 & 1 & 0 & \cdots\\
\hline
b_j > b_t & k-t & k-t & k-t-1 & k-t-1 & \cdots
\end{array}
\end{equation*}

\begin{equation*}
\begin{array}{c||c|c|c|c|c|c|c|c|c||}
& & \multicolumn{4} {|c|} {\text{\# of reference set smaller than } b_t}\\
 & \cdots & 2 \leq b_t & 1 < b_t & 1 = b_t & 0 < b_t\\
\hline
\hline
b_j < b_t & \cdots & t-2 & t-2 & t-1 & t-1\\
\hline
b_j = b_t & \cdots & 0 & 1 & 0 & 1\\
\hline
b_j > b_t & \cdots & k-t-r+3 & k-t-r+2 & k-t-r+2 & k-t-r+1\\
\end{array}
\end{equation*}

This then gives rise to the following table of multiplicities, where the lefthand column again corresponds to the number of reference elements below or equal to $b_t$.

\begin{equation*}
\begin{array}{c|c||c|c}
\multicolumn{2} {c||} {\text{Case}} & \text{Multiplicities} & \text{Mult. of }a_t\\
\hline
\hline
\multirow{7}{*}{$x_r < b_t$} & (r-1) < b_t & (t-r, k-t+1) & t-r\\
\cline{2-4}
 & (r-1) \leq b_t & (t-r+1, k-t) & t-r+1\\
\cline{2-4}
 & (r-2) < b_t & (t-r+1, k-t) & t-r+1\\
\cline{2-4}
 & \vdots & \vdots & \vdots\\
\cline{2-4}
 & 1 < b_t & (t-2, k-t-r+3) & t-1\\
\cline{2-4}
 & 1 = b_t & (t-1, k-t-r+2) & t-1\\
\cline{2-4}
 & 0 < b_t & (t-1, k-t-r+2) & t-1\\
\hline
\hline
\multirow{7}{*}{$x_r > b_t$} & (r-1) < b_t  & (t-r+1, k-t) & k-t\\
\cline{2-4}
 & (r-1) \leq b_t & (t-r+1, k-t) & k-t\\
\cline{2-4}
 & (r-2) < b_t & (t-r+2, k-t-1) & k-t-1\\
\cline{2-4}
 & \vdots & \vdots & \vdots\\
\cline{2-4}
 & 1 < b_t & (t-1, k-t-r+2) & k-t-r+2\\
\cline{2-4}
 & 1 = b_t & (t-1, k-t-r+2) & k-t-r+2\\
\cline{2-4}
 & 0 < b_t & (t, k-t-r+1) & k-t-r+1\\\end{array}
\end{equation*}

Again we can eliminate a large number of these situations. As we took $x_r$ to be the maximum of the fixed reference elements, in the first half of this table all but the top row disappear, since if $x_r$ is the largest and $x_r < b_t$ then all the other $(r-1)$ of them must also be less than $b_t$. In the second half each multiplicity is repeated twice, but we note that this doesn't affect our analysis as this is coming from separately counting the case where $b_t$ is one of our reference elements and where it isn't. This however doesn't matter, as all we are interested in is the value of $b_t$, and either way we take the answer with the larger multiplicity (i.e. the second, as $t < \frac{k}{2}$), and so recover $b_t$.\\

This works all the way down to $r = t - 1$. However when $r=t$ this no longer works, as you could be unlucky and pick as your fixed elements $S$ together with the smallest element left in the set. Then every probe would just give $x_r$ as it will be the $t^\text{th}$ element of any query. This is equivalent to noting in the above analysis that if all the first $(r-1)$ are less than $b_t$ then if $x_r < b_t$ you'll end up trying to pick the element with multiplicity 0, which doesn't exist. Hence in this case you won't be able to identify $b_t$. \\

This matches the lower bound that we expect, so we conclude that this gives an construction requiring $\binom{k+t-2}{k} + \binom{k+t-2}{t-1} \cdot \binom{n-t+1}{k-t+1}$ steps to sort the set, $\binom{k+t-2}{k}$ to find a fixed reference set of size $(t-1)$ that we can fully order and then $\binom{n-t+1}{k-t+1}$ further steps to carry out all possible queries with each possible set of $(t-1)$ fixed elements. This is $O(n^{k-t+1})$ as required.\\

\subsubsection{Adjacency algorithm}

The key concept behind the second algorithm that we present is that of knowing which elements are next to which in the ordering. We begin with the following observation, which states that this would be sufficient to determine the full ordering of the element set.

\begin{Observation}
\label{Adjgivesorder}
Let $X:= \{x_1, \ldots, x_n\}$ be a set of ordered elements, with the ordering being fixed but unknown to a user. Assume that the user knows which elements are adjacent to which others, i.e. he is given a map $\phi: X \rightarrow X^{(2)}$\footnote{Where we use the standard notation $X^{(2)}$ to be the set of all sets containing 2 elements from $X$.} such that 

\begin{equation*}
\phi(x_i) = \left\{\begin{array}{lcl}
\{x_{i-1}, x_{i+1}\} & \quad & i \notin\{1, n\}\\
\{x_2\} & & i = 1\\
\{x_{n-1}\} & & i=n
\end{array}\right.
\end{equation*}

Then the user can deduce the ordering of the element set, up to reflection.
\end{Observation}

\begin{proof}
Note that only 2 elements only have 1 neighbour, namely $x_1$ and $x_n$. Hence the user can identify this pair easily. He picks one, and assumes it is $x_1$. $x_2$ is then immediately given as the sole neighbour of $x_1$. He then proceeds iteratively -- assume he has identified $x_1, \ldots, x_r$ up to some number $1 \leqslant r < n$, $x_{r+1}$ is then the element in $\phi(x_r)$ which is not $x_{r-1}$, extending the ordering. He can repeat this until he finally finds $x_n$, and ends up either with the correct ordering or, if his initial choice of $x_1$ was incorrect, the reflection of it.
\end{proof}

Hence it suffices to find a full list of the adjacencies to determine the ordering up to reflection. Note that, having done so if the instrument is asymmetric then any single query's result will determine which of the two possible orderings the user has, and if the instrument was symmetric then this would be the best he could have hoped for anyway. As ever, a full list will not be possible, but he can try to find all the adjacencies within $X \setminus (S \cup L)$. To do this the user will eliminate possible adjacencies for each element until only the actual adjacencies remain -- and then appeal to the above observation to determine the ordering.\\

We suggest the following approach. Consider two elements from $X \setminus (S \cup L)$, $x$ and $y$ where $x$ and $y$ are adjacent. If any query returns the response $x$ and then the same query is attempted with $x$ replaced by $y$, then the second query must return the element $y$. Alternatively consider the situation where we have three elements, $a$, $b$ and $c$, all taken from $X \setminus (S \cup L)$, with $a < b < c$, and as of yet the user does not know anything about their adjacencies. Consider a query of the form $\{x_1, \ldots, x_{k-2}\} \cup \{a, b\}$ which returns the output $a$. This means that a query of $\{x_1, \ldots, x_{k-2}\} \cup \{c, b\}$, with $a$ replaced by $c$, cannot return a response of $c$ since it would pick up $b$ first. Hence if he carries out these two queries he will know that $a$ cannot be adjacent to $c$ - if it was by the first comment $c$ would have had to be the response of the second query. However the existence of $b$ between $a$ and $c$ ensures that the second query will not return $c$.\\

This motivates our approach -- the user seeks a set of queries such that, for any triple $(a, b, c)$ such that all three lie in $X \setminus (S \cup L)$ with $a < b < c$ he can find some query containing $a$ and $b$ that returns $a$, and the same query with $c$ replacing $a$. It suffices to fix some reference set of size $t-1$, which we call $y_1, \ldots, y_{t-1}$, and take all queries that contain these elements. If $a$, $b$, and $c$ are all from $X \setminus (S \cup L)$, with $a < b < c$, then as at most $t-1$ of the fixed elements $\{y_i\}$ are less than $a$, the query that consists of these reference elements, $a$, $b$, enough elements from $S$ to ensure that $a$ is the $t^\text{th}$ smallest and the remaining elements from $L$ will return $a$.\\

Provided that $a$ and $c$ are not in this fixed reference set, as we include all possible queries containing the $y_1, \ldots, y_{t-1}$, the user will also see the result of the same query with $a$ replaced by $c$, and will then be able to conclude that $a$ is not adjacent to $c$. He will be able to do this for all the remaining $c$ in $X \setminus (S \cup L)$, and thus be able to eliminate all of $a$'s non-neighbours. He will be left with $a$'s neighbours, and thus be able to deduce them.\\

If $a$ or $c$ are in the reference set, then this will not work. However we can circumvent this by carrying out three sets of queries, with disjoint fixed reference sets each time. Then, for any non-adjacent pair $a$ and $c$, one of the three sets of queries must have neither $a$ nor $c$ in its reference set. Hence in at least one set of queries the fact that $a$ and $c$ are not adjacent will be revealed. Since the user will discover this for all the elements in $X \setminus (S \cup L)$ which are not adjacent to $a$, he will be left with those that are, and thus will be able to determine the ordering using Observation \ref{Adjgivesorder}.\\

We note that this takes $3 \binom{n-(t-1)}{k - (t-1)}$, which is again $O(n^{k-(t-1)})$ queries, but with a much improved constant factor over the previous recursive structure. It is also computationally less complicated, taking approximately $n^2$ calculations to eliminate all the non-adjacencies, and then a linear number of steps to rebuild the ordering, while the previous method required potentially reconstructing all the $\binom{n}{k}$ queries, a considerably larger task.

\subsection{Multiple Output Scales}

The question now arises of which of these algorithms also works in the Multiple-Output case, where the user is given a $(k, t_1, \ldots, t_s)$ scale as in the online analogue, and asked to determine the ordering. The authors note that the recursive algorithm was computationally and conceptually complicated even in the singleton output case -- although we suspect it is possible to also use it for multiple output scales, the case analysis would make such an approach exceptionally tedious. However the adjacency based algorithm works almost immediately with almost no modifications. Since it just relies on showing that certain things would have to be included in the output if certain adjacencies existed, the same remains true even if the scale returns more elements. The only change required is that the fixed reference set is of a different size -- before it contained at most $t$ members, now it must contain at most the maximum of $t_s - 1$ and $k-t_1$ elements. This just ensures that the fixed reference elements don't take up so much of the scale that it's possible for every query containing them to only give an output consisting of members of the reference set. That established, the same analysis as before works, and so such an instrument can determine the ordering in at most the following number of queries

\begin{equation*}
3 \max\left\{\binom{n-(t_s-1)}{k-(t_s-1)}, \binom{n-(k-t_1)}{t_1}\right\}
\end{equation*}

\section{Acknowledgements}

The first author acnowledges support through funding from NSF grant DMS~1301614 and MULTIPLEX grant no. 317532. The second author was supported by the  Balassi Institute, the Fulbright Commission, and the Rosztoczy Foundation. We are both grateful to Paul Balister for his careful proofreading and useful suggestions for improvements to this paper.


\section{Bibliography}

\begin{thebibliography}{99}

\bibitem{HKL10} D. Hannasch, S-J. Kim and I. McLaughlin,
Sorting with $k$-ary Comparisons,
\textit{University of Urbana-Champaign, Illinois, REGS programme} (2010), http://www.math.illinois.edu/REGS/reports10/HanKimMc10.pdf

\end{thebibliography}
\end{document}